\documentclass[12pt]{article}
\usepackage{fullpage, amsmath, amssymb, amsfonts, amscd, amsxtra, mathrsfs, stmaryrd, amsthm, url, hyperref}
\usepackage{pictexwd,dcpic}
\usepackage[all]{xy}

\newtheorem{theorem}{Theorem}[section]
\newtheorem{lemma}[theorem]{Lemma}
\newtheorem{cor}[theorem]{Corollary}

\newtheorem{prop}[theorem]{Proposition}
\theoremstyle{definition}
\newtheorem{defn}[theorem]{Definition}

\newtheorem{convention}[theorem]{Convention}

\numberwithin{equation}{theorem}

\def\limproj{\mathop{\oalign{{\rm lim}\cr
\hidewidth$\longleftarrow$\hidewidth\cr}}}

\def\ZZ{{\mathbb Z}}
\def\QQ{{\mathbb{Q}}}

\def\NN{{\mathbb{N}}}
\def\Qp{{\mathbb{Q}}_p}
\def\Cp{{\mathbb{C}}_p}
\def\Zp{{\mathbb{Z}}_p}

\def\OO{\mathcal{O}}

\def\MM{\mathfrak{m}}
\def\D{\mathrm{D}}

\def\Gal{\mathrm{Gal}}
\def\rig{\mathrm{rig}}
\def\dR{\mathrm{dR}}

\def\dif{\mathrm{dif}}

\def\Sp{\mathrm{Sp}}
\def\sp{\mathrm{sp}}
\def\ho{\widehat{\otimes}}

\def\bcris{\mathbf{B}_{\mathrm{crys}}}

\def\bdR{\mathbf{B}_{\rm dR}}
\def\m{(\varphi,\Gamma)}
\def\dcris{\mathrm{D}_{\mathrm{crys}}}
\def\ddR{\mathrm{D}_{\mathrm{dR}}}

\def\ra{\rightarrow}
\newcommand{\btdag}[1]{\widetilde{\mathbf{B}}^{\dagger #1}}

\newcommand{\bdag}[1]{\mathbf{B}^{\dagger #1}}

\newcommand{\brig}[2]{\mathbf{B}^{\dagger #1}_{\mathrm{rig} #2}}

\title{On Properness of the Eigencurve}
\author{%
Hansheng Diao\\
       Department of Mathematics\\
       Harvard University\\
       hansheng@math.harvard.edu
\and
       Ruochuan Liu\\
       Beijing International Center\\
        for Mathematical Research\\
       Peking University\\
       liuruochuan@math.pku.edu.cn}

\begin{document}
\maketitle
\begin{abstract}
We prove that the Coleman-Mazur eigencurve is proper (over the weight space) at a large class of points.
\end{abstract}


\section{Introduction}
The $p$-adic eigencurve $\mathcal{C}$ is originally constructed by Coleman and Mazur in \cite{CM98}. It is a rigid analytic curve parameterizing finite slope overconvergent $p$-adic eigenforms of tame level 1. The construction is later generalized to any tame level $N$ by Buzzard in \cite{Buz07}. In the past decade, many efforts have been made towards understanding the geometry of the eigencurve (e.g., \cite{BuCa05}, \cite{BK05}, \cite{BuCa06}, \cite{Cal08}, and \cite{Bel12}). However, one fundamental question remains open. In \cite{CM98}, Coleman and Mazur ask the following question:  does there exist a $p$-adic family of finite slope overconvergent eigenforms, parameterized by a punctured disk, that converges to an overconvergent eigenform at the puncture which is of infinite slope?

In other words, adopting the formulation of Buzzard and Calegari (\cite{BuCa05}), this is to ask whether the projection map $\pi:\mathcal{C}\rightarrow \mathcal{W}$ satisfies the \emph{valuative criterion for properness}. In this sense, the properness of the eigencurve is proved in the case of $p=2, N=1$ in \cite{BuCa05} and is proved at the integral weights (for any $N$) in \cite{Cal08}. We want to point out that the map $\pi$ is actually not proper in the sense of rigid analytic geometry as the eigencurve $\mathcal{C}$ is of infinite degree over the weight space $\mathcal{W}$.

In this paper, we will prove that the eigencurve is proper at a large class of points.

To state our main result, we first introduce the following notion: if $\lambda\in\Zp$ and $m$ is a positive integer, define $\lambda^{\{m\}}$ to be the unique integer in $\{0, 1,\dots,p^{m}-1\}$ congruent to $\lambda$ modulo $p^m$.
For $\lambda\in\overline{\mathbb{Z}}_p$, we say that $\lambda$ satisfies condition $(*)$ if it satisfies one of the following three disjoint conditions:
\begin{enumerate}
\item[(1)]$\lambda\notin \mathbb{Z}_p$;
\item[(2)]$\lambda\in\mathbb{Z}_p\setminus \mathbb{N}$ and $\lim_{m\ra\infty} \lambda^{\{m\}}/m=\infty$;
\item[(3)]$\lambda\in\mathbb{N}$.
\end{enumerate}
Note that the condition $(*)$ is slightly stronger than being $p$-adic non-Liouville numbers (see \cite[Definition 13.1.2]{Ke07} for the definition of $p$-adic Liouville numbers). In fact, by \cite[Proposition 13.1.4]{Ke07}, we see that for $\lambda\in\mathbb{Z}_p$, it satisfies $(*)$ if and only if it is either an integer or of type 1 in the sense of $p$-adic differential equations.
It is also clear that the subset of numbers which are excluded by $(*)$ has measure $0$ in $\overline{\mathbb{Z}}_p$.

Let $\Sigma$ be the finite set of places of $\mathbb{Q}$ consisting of the infinite place and the primes dividing $pN$. For a $p$-modular representation $\overline{V}$ of $G_{\mathbb{Q}, \Sigma}$, let $\widehat{R}_{\overline{V}}$ be the universal deformation ring of the pseudo-representation associated to $\overline{V}$. Let $\widehat{R}=\prod \widehat{R}_{\overline{V}}$ where $\overline{V}$ runs through all (finitely many) $p$-modular representations.
One can regard the eigencurve $\mathcal{C}$ of tame level $N$ as an analytic subspace of
$\Sp(\widehat{R}[1/p])\times \mathbb{G}_m$. Let $\widehat{R}^\circ=\prod \widehat{R}_{\overline{V}}$ where $\overline{V}$ runs through all irreducible $p$-modular representations whose restrictions on $G_{\mathbb{Q}_p}$ are not isomorphic to twists of $\left(
\begin{smallmatrix}
 1&* \\
 0&\overline{\chi}
\end{smallmatrix}
\right)$. Here $\chi$ denotes the cyclotomic character.
Set $\mathcal{C}^\circ=\mathcal{C}\cap(\Sp(\widehat{R}^\circ[1/p])\times \mathbb{G}_m)$.

The main result of this paper is as follows.
\begin{theorem}\label{main}
Let $D$ be the closed unit disk over some finite extension $L$ over $\Qp$, and let $D^{\ast}$ be the punctured disk with the origin removed. Suppose $f: D^{\ast}\rightarrow \mathcal{C}^\circ$ is a morphism of rigid analytic spaces such that $\pi\circ f$ extends to $D$.
Moreover, suppose that the Hodge-Tate weight of $(\pi\circ f)(0)$ satisfies $(*)$.
Then the map $f$ uniquely extends to a map $\widetilde{f}: D\rightarrow \mathcal{C}^\circ$ compatible with $\pi\circ f$.
\end{theorem}

Our approach is Galois theoretical. By \cite{CM98}, the family of overconvergent eigenforms on the eigencurve $\mathcal{C}$ give rise to a family of $G_{\mathbb{Q}}$-representations on $\mathcal{C}^\circ$. Pulling back along $f$, we obtain a family of $G_{\mathbb{Q}}$-representations $V_{D^{\ast}}$ on the punctured disk $D^{\ast}$. A compactness argument shows that $V_{D^{\ast}}$ naturally extends to a family of $G_{\mathbb{Q}}$-representations $V_D$ on the entire disk. It is not difficult to see that the specialization $V_0$ of $V_D$ at the puncture is promodular in the sense of \cite{E11}. Thanks to [Corollary 1.2.2, \emph{loc.cit.}], it reduces to show that $V_0$ is trianguline.

The remaining work is purely local. Let $\alpha\in\mathcal{O}(D^{\ast})^{\times}$ be the pullback of the $U_p$-eigenvalue. It is straightforward to see that $\alpha$ extends to an analytic function on the entire disk. We will show that $V_0$ admits nonzero crystalline periods with Frobenius eigenvalue $\alpha(0)$. (Thus $\alpha(0)$ is nonzero!)

The first step is to construct the de Rham periods. In fact, we will prove the following key proposition.

\begin{prop}
Let $V_D$ be the family of $p$-adic representations as above. Suppose that the Hodge-Tate weight of
$(\pi\circ f)(0)$ satisfies $(*)$. After shrinking the disk if necessary, there exists an increasing sequence of positive integers $k_0<k_1<\cdots$ such that
\begin{enumerate}
\item[(i)] The module $(\D^+_{\dif}(V_D)/(t^{k_i}))^{\Gamma}$ is free of rank $1$ for all $i\geq 0$.
\item[(ii)] The natural map
\[
(\D^+_{\dif}(V_D)/(t^{k_{i+1}}))^{\Gamma}\rightarrow(\D^+_{\dif}(V_D)/(t^{k_i}))^{\Gamma}
\]
is an isomorphism for all $i\geq 0$.
Consequently,
\[
\D^+_{\dR}(V_S)=\D^+_{\dif}(V_S)^{\Gamma}=\varprojlim (\D^+_{\dif}(V_S)/(t^{k_i}))^{\Gamma}
\]
is free of rank $1$.
\end{enumerate}
\end{prop}

\bigskip

To prove this proposition, we make use of the results of the second author \cite{L12}.  We first show that the \textit{finite slope subspace} of $D^*$ with respect to the pair $(D^\ast,\alpha)$ is $D^*$ itself. This allows us to compare $\D^{\dagger}_{\rig}(V_{D^{\ast}})^{\varphi=\alpha, \Gamma=1}$ with the coherent modules  $(\D^+_{\dif}(V_{D^{\ast}})/(t^k))^{\Gamma}$. More precisely, for any affinoid subdomain $M(R)$ of $D^{\ast}$
and $k>\log_p|\alpha^{-1}|_{\mathrm{sp}}$ ($|\cdot|_{\mathrm{sp}}$ denotes the spectral norm on $M(R)$), the natural map
\[
\D^{\dagger}_{\rig}(V_{R})^{\varphi=\alpha, \Gamma=1}\ra (\D^+_{\dif}(V_R)/(t^k))^{\Gamma}
\]
is an isomorphism.
On the other hand, one can show that the module $\D^{\dagger}_{\rig}(V_{D^{\ast}})^{\varphi=\alpha, \Gamma=1}$ is locally free of rank 1 on the punctured disk. Hence so is $(\D^+_{\dif}(V_{R})/(t^k))^{\Gamma}$, and it follows that
the natural map
\[
(\D^+_{\dif}(V_R)/(t^{k+1}))^{\Gamma}\rightarrow(\D^+_{\dif}(V_R)/(t^k))^{\Gamma}
\]
is an isomorphism.

In order to pass to the puncture, we use a glueing argument. Note that the obstruction of lifting the elements of $(\D^+_{\dif}(V_D)/(t^{k_i}))^{\Gamma}$ to $(\D^+_{\dif}(V_D)/(t^{k_{i+1}}))^{\Gamma}$ lies at those points who have integral weights $k\in [k_i,\dots,k_{i+1}-1]$. To avoid this obstruction, our strategy is to find a smaller disk centered at the puncture satisfying the following two conditions:
\begin{enumerate}
\item[(a)]such $k$ does not occur as a weight within this smaller disk;
\item[(b)]away from this smaller disk, $k_i$ is bigger than the valuation of $\alpha$.
\end{enumerate}
Condition $(a)$ implies that the projection $(\D^+_{\dif}(V_D)/(t^{k_{i+1}}))^{\Gamma}\ra(\D^+_{\dif}(V_D)/(t^{k_{i}}))^{\Gamma}$
is an isomorphism in the smaller disk.
By condition $(b)$, we see that away from this smaller disk, both $(\D^+_{\dif}(V_D)/(t^{k_i}))^{\Gamma}$ and $(\D^+_{\dif}(V_D)/(t^{k_{i+1}}))^{\Gamma}$ are naturally isomorphic to $\D^{\dagger}_{\rig}(V_D)^{\varphi=\alpha, \Gamma=1}$. Consequently, the obstruction goes away in two different ways. However, to have this strategy work, one is led to the condition $(*)$ on the Hodge-Tate weight of $(\pi\circ f)(0)$.

Finally, applying the results of $\cite{L12}$, we know that the de Rham periods are actually crystalline with Frobenius eigenvalue $\alpha$ on the punctured disk. Thus $\dcris^+(V_D)^{\varphi=\alpha}$ is also free of rank $1$. This gives rise to the desired crystalline periods at the puncture.

\subsection*{Plan of the paper}

In section 2, we review the construction of the $\m$-module functor $\D^{\dag}_{\rig}$ and the functor $\D^+_{\dif}$ on families of $p$-adic Galois representations. In section 3, we introduce the notion of finite slope subspaces and recall some of the results from \cite{L12}. Section 4 contributes to the proof of the main result. In section 4.1, we show that the finite slope subspace of the
families of $p$-adic representations on the punctured disk is the punctured disk itself.
In section 4.2, we work out the glueing argument in detail and prove Theorem ~\ref{main}.

\subsection*{Notations}
Let $p$ be a prime number. Fix an algebraic closure $\overline{\QQ}_p$ of $\Qp$ and let $\Cp$ be the completion of $\overline{\QQ}_p$ with respect to the normalized $p$-adic valuation $v_p$ such that $v_p(p)=1$. Let $K$ be a finite extension of $\Qp$ in $\overline{\QQ}_p$. Let $\OO_K$ be the ring of integers of $K$ and let $\pi_K$ be a uniformizer. Let $K_0$ denote the maximal unramified subextension of $K$. Choose a compatible system of primitive $p$-power roots of unity $(\zeta_{p^n})_{n\geq 0}$; i.e., each $\zeta_{p^n}$ is a primitive $p^n$-th root of 1, and $\zeta_{p^{n+1}}^p=\zeta_{p^n}$ for all $n\geq 0$. Write $K_n=K(\zeta_{p^n})$ for any $n\geq 1$ and let $K_{\infty}=\cup_{n\geq 1}K_n$. Let $K'_0$ denote the maximal unramifed extension of $K_0$ inside $K_{\infty}$. Moreover, we write $G_K=\Gal(\overline{K}/K)$, $H_K=\Gal(\overline{K}/K_{\infty})$, and $\Gamma=\Gamma_K=\Gal(K_{\infty}/K)$.

We adopt the notations of $p$-adic Hodge theory as in the standard literature (e.g., \cite{Ber04}). As for the variations of Fontaine's ``boldface $\mathbf{B}$-rings'', we adopt Berger and Colmez's language. In particular, we define
\[\btdag{,s}, \bdag{,s}_K, \bdag{}_K, \brig{,s}{,K}, \brig{}{,K}\]
as in \cite{Ber02} or \cite{BC07}. It is worth mentioning that for $s$ sufficiently large, the ring $\brig{,s}{,K}$ coincides with the ring $\mathcal{R}^s_{K'_0}$ in \cite{KL10}. Hence $\brig{}{,K}$ coincides with the Robba ring $\mathcal{R}_{K'_0}$.

For a reduced affinoid algebra $S$ over $\Qp$, let $|\cdot|_{\mathrm{sp}}$ be the spectral norm and $\textrm{val}_S$ the corresponding valuation. Let $\OO_S$ be the ring of integers with respect to $\textrm{val}_S$. By an $S$\textit{-linear representation} of $G_K$, we mean a free $S$-module $V_S$ of finite rank equipped with a continuous $S$-linear action of $G_K$. If $M(R)\subset M(S)$ is an affinoid subdomain, we write $V_R$ for the base change of $V_S$ from $S$ to $R$. For each point $x\in M(S)$, we write $V_x=V_S\otimes_S k(x)$ for the specialization of $V_S$ at $x$.

More generally, for a separated rigid analytic space $X$ over $\Qp$, by a \emph{family of $p$-adic representations} on $X$ we mean a locally free coherent $\mathcal{O}_X$-module equipped with a continuous $\mathcal{O}_X$-linear $G_K$-action.

\subsection*{Acknowledgements}
The authors would like to thank Kiran Kedlaya, Mark Kisin, and Liang Xiao for useful comments on earlier drafts of this paper.

\section{Families of Galois representations}
We start with a brief review on the theory of $\m$-module functor $\D^{\dagger}_{\rig}$ for families
of $p$-adic Galois representations. This theory is first introduced by Berger-Colmez, and later generalized to
Robba rings by Kedlaya-Liu. We do not attempt to present a self-contained survey on the subject. We
refer the reader to \cite{BC07} and \cite{KL10} for more details.

\subsection{The module $\D^{\dag}_{\rig}(V_S)$}
The classical theory of $\m$-modules is introduced by Fontaine over $\mathbf{B}_K$, Cherbonnier-Colmez over $\bdag{}_K$, and Berger over $\mathbf{B}_{\rig,K}^\dagger$. In particular, there are equivalences of categories between $p$-adic representations of $G_K$ and categories of \'etale $\m$-modules over $\mathbf{B}_K$, $\bdag{}_K$, and $\brig{}{,K}$, respectively.

In Berger and Colmez's paper \cite{BC07}, they construct the overconvergent $\m$-module functor $S$-linear representations. More precisely, let $V_S$ be an $S$-linear representation of $G_K$ of rank $d$. For $s\geq s(V_S)$, one can construct a locally free $S\ho_{\Qp}\bdag{,s}_K$-module $\D^{\dag, s}_K(V_S)$ of rank $d$ such that there is a natural isomorphism
\[
(S\ho_{\Qp}\btdag{,s})\otimes_{S\ho\bdag{,s}_K}\D^{\dag, s}_K(V_S)\xrightarrow{\sim} (S\ho_{\Qp}\btdag{,s})\otimes_S V_S
\]
and, for any $x\in M(S)$, the specialization map $S/\MM_x\otimes_S \D^{\dag,s}_K(V_S)\rightarrow \D^{\dag,s}_K(V_x)$ is an isomorphism.
We put \[S\ho_{\Qp}\bdag{}_K=\bigcup_{s>0}S\ho_{\Qp}\bdag{,s}_K,\]
and define
\[\D^{\dag}_K(V_S)=(S\ho_{\Qp}\bdag{}_K)\otimes_{S\ho_{\Qp}\bdag{,s}_K}\D^{\dag, s}_K(V_S)=\bigcup_{s>s(V_S)}\D^{\dag, s}_K(V_S).\]
Clearly, the $S\ho_{\Qp}\bdag{}_K$-module $\D^{\dag}_K(V_S)$ is locally free of rank $d$ and satisfies the base change property $S/\MM_x\otimes_S \D^{\dag}_K(V_S)\xrightarrow{\sim} \D^{\dag}_K(V_x)$. Moreover, $\D^{\dag}_K(V_S)$ is equipped with commuting $\varphi,\Gamma$-actions. This makes $\D^{\dag}_K(V_S)$ an \emph{\'etale $\m$-module} over $S\ho_{\Qp}\bdag{}_K$ in the sense of \cite[Definition 2.8]{KL10}. However, the functor $\D^{\dag}_K$ is far from an equivalence of categories from $S$-linear representations to \'etale $\m$-modules over $S\ho_{\Qp}\bdag{}_K$ for general $S$.

There is also a family version of $\m$-modules over the Robba rings. Recall that $\brig{,s}{,K}$ is the completion of $\bdag{,s}_K$ with respect to the Fr\'echet topology. For $s>s(V_S)$, we define \[\D^{\dag, s}_{\rig, K}(V_S)=(S\ho_{\Qp}\brig{,s}{,K})\otimes_{S\ho\bdag{,s}_K}\D^{\dag,s}_K(V_S).\]
Put \[S\ho_{\Qp}\brig{}{,K}=\bigcup_{s>0}S\ho_{\Qp}\brig{,s}{,K},\]
and define
\[\D^{\dag}_{\rig,K}(V_S)=(S\ho_{\Qp}\brig{}{,K})\otimes_{S\ho_{\Qp}\brig{,s}{,K}}\D^{\dag, s}_{\rig,K}(V_S)=\bigcup_{s>s(V_S)}\D^{\dag, s}_{\rig,K}(V_S).\]
Then $\D^{\dag}_{\rig,K}(V_S)$ is an \emph{\'etale family of $\m$-module} over $S\ho_{\Qp}\brig{}{,K}$ in the sense of \cite[Definition 6.3]{KL10}.

Sheafify the above construction on the affinoid space $M(S)$, we obtain a sheaf of $\m$-modules. More precisely, define the presheaves $\mathscr{D}^{\dag, s}_{\rig, K}(V_S)$ and $\mathscr{D}^{\dag}_{\rig, K}(V_S)$ with respect to the weak $G$-topology of $M(S)$ by setting

\[\mathscr{D}^{\dag, s}_{\rig, K}(V_S)(M(S'))=\D^{\dag, s}_{\rig, K}(V_{S'}),\,\,\,\,\,\,\,\,\mathscr{D}^{\dag}_{\rig, K}(V_S)(M(S'))=\D^{\dag}_{\rig, K}(V_{S'})\]
for any affinoid subdomain $M(S')$ of $M(S)$. Thanks to \cite[Proposition 1.2.2]{L12}, we know that the presheaves $\mathscr{D}^{\dag, s}_{\rig, K}(V_S)$ and $\mathscr{D}^{\dag}_{\rig, K}(V_S)$ are indeed sheaves.

Finally, we can extend the above construction to any separated rigid analytic space $X$ over $\mathbb{Q}_p$ by glueing the affinoid pieces. Let $V_X$ be a family of $p$-adic representations over $X$. Let $\{M(S_i)\}_{i\in I}$ be an admissible covering of $X$ by affinoid subdomains such that the restriction of $V_X$ on $M(S_i)$ are free $S_i$-representations. Define the sheaf $\mathscr{D}^{\dag}_{\rig, K}(V_X)$ on $X$ by glueing the sheaves $\mathscr{D}^{\dag}_{\rig, K}(V_{S_i})$ for $i\in I$. This construction is clearly independent of the choice of the admissible covering.

\subsection{The module $\D^+_{\dif}(V_S)$}
In this section, we review the family version of Berger's $\D^+_{\dif}$ functor introduced in $\cite{L12}$. Equip $K_n[[t]]$ with the induced Fr\'echet topology via the identification
\[
K_n[[t]]=\varprojlim K_n[[t]]/(t^k)\cong \varprojlim K_n^k\cong K_n^{\NN}.
\]
For $n$ sufficiently large, say $n\geq n(s)$, we have the localization map $\iota_n:\brig{,s}{,K}\rightarrow K_n[[t]]$ (see \cite{Ber02}), which induces a continuous map $S\ho_{\Qp}\bdag{,s}_{\rig,K}\rightarrow S\ho_{\Qp}K_n[[t]]$. Define
\[\D^{K_n,+}_{\dif}(V_S)=(S\ho_{\Qp}K_n[[t]])\otimes_{\iota_n,S\ho_{\Qp}\bdag{,s}_{\rig,K}}\D^{\dag,s}_{\rig,K}(V_S).\]
It is clear that $\D^{K_n,+}_{\dif}(V_S)$ is a locally free $S\ho_{\Qp}K_n[[t]]$-module of rank $d$.

Abusing the notation, we still denote by $\iota_n$ the natural map $\iota_n:\D^{\dag,s}_{\rig,K}(V_S)\rightarrow \D^{K_n,+}_{\dif}(V_S)$. We also define
\[\D^{K,+}_{\dif}=\bigcup_{n\geq n(s)}\D^{K_n,+}_{\dif}(V_S).\]

\convention When the base field $K$ is clear, we drop the label $K$ in all of the ``$\D$-functors". In particular, $\D^{K,+}_{\dif}$ and $\D^{K_n, +}_{\dif}$ become $\D^+_{\dif}$ and $\D^{+,n}_{\dif}$, respectively.\\

Sheafifying the above construction, we can define the following presheaves.

\begin{defn}
Let $k$ be a positive integer. Define the presheaves $\mathscr{D}^{+,n}_{\dif}(V_S)$ and $\mathscr{D}^{+, n}_{\dif}(V_S)/(t^k)$ with respect to the weak $G$-topology of $M(S)$ by setting
\[\mathscr{D}^{+,n}_{\dif}(V_S)(M(S'))=\D^{+,n}_{\dif}(V_{S'}),\,\,\,\,\,\,\,\,\,\,(\mathscr{D}^{+,n}_{\dif}(V_S)/(t^k))(M(S'))=\D^{+,n}_{\dif}(V_{S'})/(t^k)\]
for any affinoid subdomain $M(S')$ of $M(S)$. Moreover, define the presheaves

\[\mathscr{D}^+_{\dif}(V_S)=\varinjlim_{n\rightarrow \infty}\mathscr{D}^{+,n}_{\dif}(V_S)\,\,\,\,\,\textrm{and}\,\,\,\,\,\, \mathscr{D}^+_{\dif}(V_S)/(t^k)=\varinjlim_{n\rightarrow \infty}\mathscr{D}^{+,n}_{\dif}(V_S)/(t^k).\]
\end{defn}
Thanks to \cite[Proposition 1.4.2]{L12}, we know that the presheaves $\mathscr{D}^{+,n}_{\dif}(V_S)$ and $\mathscr{D}^+_{\dif}(V_S)$ are indeed sheaves. By \cite[Proposition 1.4.5]{L12}, we know that the presheaf $\mathscr{D}_{\dif}^{+,n}(V_S)/(t^k)$ is a locally free coherent sheaf on $M(S)$, and the presheaf  $(\mathscr{D}_{\dif}^{+,n}(V_S)/(t^k))^{\Gamma}$ is a coherent sheaf on $M(S)$.
Nevertheless, we can construct the sheaves $\mathscr{D}^+_{\dif}(V_X)$ and $\mathscr{D}^+_{\dif}(V_X)/(t^k)$ for any family of $p$-adic representations over a separated rigid analytic space $X$. We choose an admissible covering $\{M(S_i)\}_{i\in I}$ and glue up the sheaves $\mathscr{D}^+_{\dif}(V_{S_i})$'s (resp., $\mathscr{D}^+_{\dif}(V_{S_i})/(t^k)$'s). The resulting sheaf is independent of the choice of the covering.\\

To wrap up the section, we recap the definitions of the functors $\ddR^+$ and $\dcris^+$ in the family version. Recall that $\bdR^+/(t^k)$ is naturally a $p$-adic Banach space for any positive integer $k$. This naturally gives rise to a
Fr\'echet topology on
\[
\bdR^+=\varprojlim \bdR^+/(t^k).
\]
So we can define $S\ho_{\Qp}\bdR^+$. We can also define $S\ho_{\Qp}\bcris^+$ as $\bcris^+$ has a natural $p$-adic Banach space structure. For an $S$-linear representation $V_S$ of $G_K$, we set
\[\ddR^+(V_S)=((S\ho_{\Qp}\bdR^+)\otimes_S V_S)^{G_K},\] and \[\dcris^+(V_S)=((S\ho_{\Qp}\bcris^+)\otimes_S V_S)^{G_K}.\]
Just as in the field case, one can recover $\ddR^+(V_S)$ and $\dcris^+(V_S)$ from $\D^+_{\dif}(V_S)$ and $\D^{\dag}_{\rig}(V_S)$. By \cite[Lemma 4.3.1]{BC07}, we know that $\ddR^+(V_S)=\D_{\dif}^+(V_S)^{\Gamma}$. By \cite[Theorem 1.1.1]{Be13}, we have $\dcris^+(V_S)=(\D^{\dagger}_{\rig}(V_S))^{\Gamma}$.


\section{Finite slope subspaces}
From now on, we take $K=\Qp$. Let $X$ be a separated and reduced rigid analytic space over some finite extension $L$ of $\Qp$. Let $V_X$ be a family of $p$-adic representations of $G_{\Qp}$ over $X$. We further assume that the Sen polynomial for $V_X$ is of the form $TQ(T)$ for some $Q(T)\in\mathcal{O}_X[T]$. Let $\alpha\in \mathcal{O}(X)^\times$ be an invertible analytic function on $X$.

Let us clarify one piece of notation. If $X'\subset X$ is an analytic subspace and $h$ is an analytic function on $X$, we write $X'_h$ for the non-vanishing locus of $h$ in $X'$. In particular, if $j$ is an integer, then $X'_{Q(j)}$ excludes exactly those points which has $-j$ as a Hodge-Tate weight.

\begin{defn}\label{finiteslopesubspace}
Let $(X, V_X, \alpha)$ be a triple as above. An analytic subspace $X_{fs}\subset X$ is called a \textit{finite slope subspace} of $X$ with respect to the pair $(V_X, \alpha)$ if it satisfies the following two conditions:
\begin{enumerate}
\item[(1)] For every integer $j\leq 0$, the subspace $(X_{fs})_{Q(j)}$ is Zariski-dense in $X_{fs}$.
\item[(2)] For any affinoid algebra $R$ over $L$ and any morphism $g:M(R)\rightarrow X$ which factors through $X_{Q(j)}$ for every integer $j\leq 0$, the morphism $g$ factors through $X_{fs}$ if and only if the natural map
\[\iota_n:(\D^{\dag}_{\rig}(V_R))^{\varphi=g^{\ast}(\alpha), \Gamma=1}\rightarrow (\D^{+, n}_{\dif}(V_R))^{\Gamma}\]
is an isomorphism for all sufficiently large $n$.
\end{enumerate}
\end{defn}

This definition is inspired by the notion of finite slope subspace introduced by Kisin in \cite{Kis03}. The key difference is that we can remove the ``$Y$-smallness'' condition imposed in Kisin's original definition. The following result guarantees the existence of finite slope subspace.

\begin{theorem}[\cite{L12}, Theorem 3.3.1]
Given any triple $(X, V_X, \alpha)$ as above, the rigid analytic space $X$ has a unique finite slope subspace $X_{fs}$.
\end{theorem}

The following result is going to play an important role in next section.

\begin{theorem}[\cite{L12}, Theorem 3.3.4]\label{mainlemma1}
Let $M(S)$ be an affinoid subdomain of $X_{fs}$ and let $|\cdot|_{\sp}$ denote the spectral norm taken on $S$. Then for any $n\geq n(V_S)$ and $k>\log_p|\alpha^{-1}|_{\sp}$, the natural map of sheaves
\[(\mathscr{D}^{\dag}_{\rig}(V_S))^{\varphi=\alpha, \Gamma=1}\rightarrow (\mathscr{D}^{+, n}_{\dif}(V_S)/(t^k))^{\Gamma}\]
is an isomorphism. As a consequence, we know that $(\mathscr{D}^{\dag}_{\rig}(V_{X_{fs}}))^{\varphi=\alpha, \Gamma=1}$ is a coherent sheaf on $X_{fs}$
\end{theorem}


\section{Proof of the main result}




We keep the setup as in the introduction. We further assume that the map $\pi\circ f:D\rightarrow\mathcal{W}$ is dominant. Otherwise the desired statement is obvious. We normalize the weight function on $\mathcal{C}$ so that a classical modular form of weight $k$ has weight $k-1$. Let $\kappa\in\mathcal{O}(D)$ be the pullback of the weight function via the map $\pi\circ f$; in particular, we have $\kappa(0)\in\mathcal{O}_L$.  Let $\alpha:D^{\ast}\rightarrow \mathbb{G}_m$ be the pullback of the $U_p$-eigenvalue.

We claim that $V_{D^{\ast}}$ extends uniquely to a family of $p$-adic representations $V_D$ on the entire disk $D$. Indeed, the family on the punctured disk gives a continuous map from the universal deformation ring
$\widehat{R}^\circ_{\overline{V}}$ to $\mathcal{O}(D^{\ast})$.
Since $\widehat{R}^\circ_{\overline{V}}$ is compact, the image of this map lives in the maximal compact subgroup of
$\mathcal{O}(D^{\ast})$, which is exactly $\OO_S=\OO_L\langle T\rangle$.
This yields that the family extends to a $G_{\mathbb{Q}}$-representation $V_D$ on $D$.


\subsection{Finite slope subspace of the punctured disk}

We start with recalling the following result from \cite{L12}.


\begin{prop}[\cite{L12}, Proposition 5.1.4]\label{prop:fs-C}
The finite slope subspace of $\mathcal{C}$ with respect to $(V_{\mathcal{C}}, \alpha_{\mathcal{C}})$ is $\mathcal{C}$ itself.
\end{prop}

\begin{cor}
The finite slope subspace $(D^{\ast})_{fs}$ of the punctured disk $D^{\ast}$ associated to $(V_{D^{\ast}}, \alpha)$ is $D^{\ast}$ itself.
\end{cor}

\begin{proof}
We need to check that $D^{\ast}$ satisfies the conditions (1) and (2) in Definition~\ref{finiteslopesubspace}. Since $f$ is dominant and the subspace of points with non-integral weights is Zariski dense in $\mathcal{C}$, we deduce that the subspace of points with non-integral weights is Zariski dense in $D^{\ast}$. This implies that $D^\ast$ satisfies the condition (1).

Now we check condition (2). We have to show the following: for any affinoid algebra $R$ and any morphism $g:M(R)\rightarrow D^{\ast}$ that factors though $D^{\ast}_{Q(j)}$ for all $j\leq 0$, the natural map
\[ \D^{\dag}_{\rig}(V_R)^{\varphi=g^{\ast}\alpha, \Gamma=1}\rightarrow \D^{+, n}_{\dif}(V_R)^{\Gamma}\]
is an isomorphism for all $n$ sufficiently large. Note that $g$ factors though $D^{\ast}_{Q(j)}$ if and only if $f\circ g$ factors though $\mathcal{C}_{Q(j)}$. By Proposition \ref{prop:fs-C}, we have
\[
\D^{\dag}_{\rig}(V_R)^{\varphi=g^{\ast}\alpha, \Gamma=1}=\D^{\dag}_{\rig}(V_R)^{\varphi=g^{\ast}f^{\ast}\alpha_{\mathcal{C}}, \Gamma=1}\cong \D^{+, n}_{\dif}(V_R)^{\Gamma}
\]
for all $n$ sufficiently large.
\end{proof}

\begin{prop}\label{isomLiu}
For any affinoid subdomain $M(R)$ of the puncture closed unit disk and any integer $k>\log_p|\alpha^{-1}|_\mathrm{sp}$ (the spectral norm is taken over $M(R)$),
the natural map
\[
(\mathscr{D}^{\dag}_{\rig}(V_R))^{\varphi=\alpha,\Gamma=1}\rightarrow (\mathscr{D}_{\dif}^{+}(V_R)/(t^k))^{\Gamma}
\]
is an isomorphism. In particular, $(\mathscr{D}_{\dif}^{+}(V_R)/(t^k))^{\Gamma}$ is locally free of rank $1$.
\end{prop}

\begin{proof}
Since the finite slope subspace of the punctured disk is itself, the isomorphism follows immediately form Theorem~\ref{mainlemma1}. Since $D^\times$ is a smooth curve, we may apply the argument of \cite[Theorem 5.4.3]{L12} to deduce that $\mathscr{D}^{\dag}_{\rig}(V_R))^{\varphi=\alpha,\Gamma=1}$ is locally free of rank 1.
\end{proof}


\subsection{Proof of Theorem \ref{main}}


We first introduce a piece of notation: for any $n\in\mathbb{N}$, let $S_{n}=L\langle p^{-n}T\rangle$, and let $S_{n,n'}=L\langle p^{-n}T, p^{n'}T^{-1}\rangle$ for $n'>n$. To simplify the notation, we write $V_n, V_{n, n'}$ for $V_{S_n}$ and $V_{S_{n, n'}}$, respectively. Before proceeding, let us make the following observation on the growth rate of $\alpha$.

\begin{lemma}\label{lem:alpha-val}
There exist positive integers $N$ and $c
_0$ such that for any $x\in M(S_N)$, $v_p(\alpha(x))\leq c_0v_p(x)$.
\end{lemma}
\begin{proof}
Since $\alpha$ is bounded in $D^\ast$, it extends to a rigid analytic function on the entire disk. By looking at the holomorphic expansion of $\alpha$ at the origin, the assertion becomes clear.
\end{proof}

From now on, we assume that $\kappa(0)$ satisfies $(*)$.

\begin{lemma}\label{sequence}
There exists an increasing sequence of positive integers $\{k_i\}_{i\in\mathbb{N}}$ such that, for each $i$,
\begin{enumerate}
\item[(a)] $k_i-1> c_0(2+\log_p k_{i+1})$, and
\item[(b)] there is no integer in $[k_{i}, k_{i+1}]$ congruent to $\kappa(0)$ modulo $p^{\lceil\log_p k_{i+1}\rceil}$.
\end{enumerate}
\end{lemma}

\begin{proof}
Note that for $m$ sufficiently large, say $m\geq M_0$, we have $p^{m}-1> c_0(m+3)$.

If $\kappa(0)\not\in\Zp$, then for sufficiently large $m$, there is no integer congruent to $\kappa(0)$ modulo $p^m$. So one can just take $k_1$ to be sufficiently large and $k_{i+1}=k_i+1$.

If $\kappa(0)\in\mathbb{N}$, we take a positive integer $m\geq M_0$ such that $p^m>\kappa(0)$, and we set $k_i=p^{m+i}$.

Now we assume $\kappa(0)\in\Zp \setminus \mathbb{N}$. Since $\kappa(0)$ satisfies $(*)$, there exists some integer $m'\geq \max\{N,3c_0+1\}$ such that $\kappa(0)^{\{m\}}\geq (c_0+1)m$ for any $m\geq m'$. Let $m_0<m_1<m_2<\cdots$ be all the $m\geq m'$ where $\kappa(0)^{\{m\}}$ jumps (i.e. all the $m\geq m'$ such that $\kappa(0)^{\{m\}}\neq \kappa(0)^{\{m-1\}}$). Set $k_i=\kappa(0)^{\{m_i\}}$.  Then, clearly, $m_i=\lceil\log_p k_i\rceil$ and there is no integer in $[k_{i-1}, k_i-1]$ congruent to $\kappa(0)$ modulo $p^{m_i}$. It remains to show $k_i-1> c_0(2+\log_p k_{i+1})$. Note that $\kappa(0)^{\{m_{i+1}-1\}}=\kappa(0)^{\{m_i\}}=k_i$. It follows
\[
k_i\geq (c_0+1)(m_{i+1}-1)> c_0m_{i+1}+m'-c_0\geq c_0m_{i+1}+2c_0+1=c_0(2+\lceil \log_p k_{i+1}\rceil)+1,
\]
yielding that $\{k_i\}_{i\in\mathbb{N}}$ satisfies $(a)$.
\end{proof}

We write $M=\lceil\log_p k_1\rceil-1$. By cutting out the small terms, we may assume that $M\geq N$.
In the rest of this subsection, let $\{k_i\}_{i\in\mathbb{N}}$ be as given in the above Lemma.

\begin{prop}\label{rank1}
The coherent $S_{M}$-modules $(\D_\dif^+(V_M)/(t^{k_i}))^\Gamma$ are all free of rank 1.
\end{prop}
\begin{proof}
Note that $S_M$ is a PID and $(\D_\dif^+(V_M)/(t^{k_i}))^\Gamma$ is a finite torsion free $S$-module; thus it is a finite free $S$-module. On the other hand, we deduce from Lemma \ref{lem:alpha-val} and Lemma \ref{sequence} that
$k_i\geq c_0(M+1)\geq\log_p|\alpha^{-1}|_{\mathrm{sp}}$ on $M(S_{M, M+1})$. Hence $(\D_\dif^+(V_{M,M+1})/(t^{k_i}))^\Gamma$ is a locally free $S_{M,M+1}$-module of rank 1 by Proposition ~\ref{isomLiu}. That is, the base change of $(\D_\dif^+(V_{M})/(t^{k_i}))^\Gamma$ to $S_{M,M+1}$ is locally free of rank 1, yielding that $(\D_\dif^+(V_{M})/(t^{k_i}))^\Gamma$ is a free $S_M$-module of rank 1.
\end{proof}

Remember that we want to extend the crystalline periods to the entire disk. The main idea, however, is first lifting to de Rham periods, which is the inverse limit of all $(\D_\dif^+(V_M)/(t^{k_i}))^\Gamma$. So we need to understand the connection maps
\[
(\D_\dif^+(V_M)/(t^{k_{i+1}}))^\Gamma\ra(\D_\dif^+(V_M)/(t^{k_i}))^\Gamma.
\]

\begin{lemma}\label{isom}
Let $M(R)$ be an affinoid subdomain of the closed unit disk, and let $k$ be a positive integer. If for any $x\in M(R)$, $\kappa(x)$ is not equal to $k$, the natural map
\[(\D_{\dif}^+(V_R)/(t^{k+1}))^{\Gamma}\rightarrow (\D_{\dif}^+(V_R)/(t^k))^{\Gamma}\]
is an isomorphism.
\end{lemma}

\begin{proof}
This follows directly from \cite[Corollary 1.5.3]{L12}.
\end{proof}

\begin{prop}\label{prop:connection-map}
For each $i$, the natural map
\[
(\D_\dif^+(V_M)/(t^{k_{i+1}}))^\Gamma\ra(\D_\dif^+(V_M)/(t^{k_i}))^\Gamma
\]
is an isomorphism.
\end{prop}
\begin{proof}
Note that $\lfloor (k_i-1)/c_0\rfloor> \lceil \log_p k_{i+1}\rceil$. Therefore it suffices to show that both
\begin{equation}\label{eq:map-1}
(\D_\dif^+(V_{M,\lfloor (k_i-1)/c_0\rfloor})/(t^{k_{i+1}}))^\Gamma\ra(\D_\dif^+(V_{M,\lfloor (k_i-1)/c_0\rfloor})/(t^{k_i}))^\Gamma
\end{equation}
and
\begin{equation}\label{eq:map-2}
(\D_\dif^+(V_{\lceil\log_p k_{i+1}\rceil})/(t^{k_{i+1}}))^\Gamma\ra(\D_\dif^+(V_{\lceil\log_p k_{i+1}\rceil})/(t^{k_i}))^\Gamma
\end{equation}
are isomorphisms. For $x\in M(S_{M,\lfloor (k_i-1)/c_0\rfloor})$, we have
\[
k_{i+1}>k_i>c_0\lfloor (k_i-1)/c_0\rfloor\geq v_p(\alpha(x)).
\]
Hence both sides of (\ref{eq:map-1}) are naturally isomorphic to $(\D_\rig^\dag(V_{M,\lfloor(k_i-1)/c_0\rfloor}))^{\varphi=\alpha,\Gamma=1}$, yielding that (\ref{eq:map-1}) is an isomorphism.

Note that $\kappa(x)$ is congruent to $\kappa(0)$ modulo $p^{\lceil\log_p k_{i+1}\rceil}$ for any $x\in M(S_{\lceil\log_p k_{i+1}\rceil})$. Thus we deduce from the construction of $\{k_i\}_{i\in\mathbb{N}}$ that there is no $x\in M(S_{\lceil\log_p k_{i+1}\rceil}) $ such that $\kappa(x)$ belong to $\{k_i,\cdots, k_{i+1}-1\}$. Hence (\ref{eq:map-2})
is an isomorphism by the previous lemma.

\end{proof}

Now we are ready to prove Theorem \ref{main}. Since the eigencurve $\mathcal{C}$ is the Zariski closure of the set of classical points in $\Sp(\widehat{R}[1/p])\times \mathbb{G}_m$, it follows that $V_0$ lies in the Zariski closure of the set of classical points in $\Sp(\widehat{R}[1/p])$. Hence it is promodular in the sense of \cite{E11}. By [Corollary 1.2.2, \emph{loc.cit.,} ], it suffices to show that $V_0$ has nonzero crystalline periods with Frobenius eigenvalue $\alpha(0)$.
We first deduce from Proposition \ref{rank1} and Proposition \ref{prop:connection-map} that
$(\D_\dif^+(V_M))^\Gamma=\limproj(\D_{\dif}^{+}(V_M)/(t^{k_i}))^{\Gamma}$ is a free $S_M$-module of rank 1.

Using Proposition ~\ref{isomLiu},
we get that the natural map
\[
\D_\rig^\dag(V_{M,M+1})^{\varphi=\alpha,\Gamma=1}\ra(\D_{\dif}^{+}(V_{M,M+1}))^{\Gamma}
\]
is an isomorphism. Hence $\D_{\mathrm{dR}}^{+}(V_{M,M+1})=\D^+_{\mathrm{crys}}(V_{M,M+1})^{\varphi=\alpha}$.
It follows that
\[
(\D_{\dif}^{+}(V_M))^{\Gamma}=\D^+_{\mathrm{dR}}(V_M)\subset V_M\widehat{\otimes}\bdR^+\cap V_{M,M+1}\widehat{\otimes}\bcris^+=V_M\widehat{\otimes}\bcris^+.
\]
The last equality follows from the lemma in the appendix. This implies that
\[
\D^+_{\mathrm{crys}}(V_M)^{\varphi=\alpha}=\D^+_{\mathrm{crys}}(V_M)=\D^+_{\mathrm{dR}}(V_M)
\]
is a free $S_M$-module of rank 1.

Finally, we choose a generator $a$ of $\D^+_{\mathrm{crys}}(V_M)^{\varphi=\alpha}$.
Note that for suitable $n\in\mathbb{N}$, the image of $T^{-n}a$ in $\D^+_{\mathrm{crys}}(V_0)$ is nonzero.
This yields a nonzero crystalline period of $V_0$ with Frobenius eigenvalue $\alpha(0)$, concluding the proof.


\bigskip
\bigskip

\section*{Appendix: A technical lemma}
The main purpose of this section is to prove the following lemma.
\begin{lemma} \label{keylemma}
We have the following identity inside Fr\'echet algebra $S_{M, M+1}\ho_{\Qp} \bdR^+$.
\[S _M\ho_{\Qp}\bdR^+\cap S_{M, M+1}\ho_{\Qp}\bcris^+=S_M\ho_{\Qp}\bcris^+. \]
\end{lemma}

First let us introduce some notations.

\begin{defn} Let $A$ be a Banach algebra over $\Qp$.
\begin{enumerate}
\item[(i)] For any $n>0$, define the Banach algebra $A\langle p^{-n}T\rangle$ to be the ring of  formal power series $\sum_{i\geq 0} a_iT^i$ with $a_i\in A$ and such that $|a_i|p^{-ni}\rightarrow 0$ as $i\rightarrow \infty$. It is equipped with a Banach norm $|\sum_{i\geq 0}a_iT^i|=\sup |a_i|p^{-ni}$.
\item[(ii)] For any $n'>n>0$, define the Banach algebra $A\langle p^{-n}T, p^{n'}T^{-1}\rangle$ to be the ring of double ended formal series $\sum_{i\in \ZZ} a_iT^i$ with $a_i\in A$ and such that $|a_i|p^{-ni}\rightarrow 0$ as $i\rightarrow\infty$ and $|a_i|p^{-n'i}\rightarrow 0$ as $i\rightarrow -\infty$. It is equipped with a Banach norm $|\sum_{i\in \ZZ}a_iT^i|=\max\{\sup|a_ip^{-ni}|, \sup|a_ip^{-n'i}|\}$.
\end{enumerate}
\end{defn}

\begin{defn}
Let $A=\varprojlim A_i$ be a Fr\'echet algebra where $A_i$'s are $\Qp$-Banach algebras. Define the Fr\'echet algebra $A\langle p^{-n}T\rangle$ to be the inverse limit of Banach algebras $A_i\langle p^{-n}T\rangle$. In particular, $A\langle p^{-n}T\rangle$ can be naturally identified with a subset of $A[[T]]$.

Similarly, define the Fr\'echet algebra $A\langle p^{-n}T, p^{n'}T^{-1}\rangle$ to be the inverse limit of Banach spaces $A_i\langle p^{-n}T, p^{n'}T^{-1}\rangle$. It can be naturally identified with a subset of $A[[T, T^{-1}]]$.
\end{defn}

\begin{lemma}\label{banach}
Let $A$ be a $\Qp$-Banach algebra. For any $n>0$, we have natural identification of Banach algebras
\[\eta_{n, A}: S_n\ho_{\Qp}A\overset{\sim}{\longrightarrow}A\langle p^{-n}T\rangle\]
Similarly, for any $n'>n>0$, we have natural identification of Banach algebras
\[\eta_{n,n',A}:S_{n,n'}\ho_{\Qp}A\overset{\sim}{\longrightarrow}A\langle p^{-n}T, p^{n'}T^{-1}\rangle.\]
Moreover, the following diagram commutes
\[\xymatrix@=45pt{
S_n\ho_{\Qp}A\ar[r]^{\eta_{n,A}} \ar@{^{(}->}[d] & A\langle p^{-n}T\rangle\ar@{^{(}->}[d]\\
S_{n,n'}\ho_{\Qp}A \ar[r]^{\eta_{n,n',A}} & A\langle p^{-n}T, p^{n'}T^{-1} \rangle
}
\]
where the vertical arrows are natural inclusions.
\end{lemma}

\begin{proof}
This is obvious.
\end{proof}

\begin{lemma}\label{frechet}
Let $A=\varprojlim A_i$ be a $\Qp$-Fr\'echet algebra where $A_i$ are $\Qp$-Banach algebras. For any $n>0$, we have natural identification of Fr\'echet algebras
\[\eta_{n, A}: S_n\ho_{\Qp}A\overset{\sim}{\longrightarrow}A\langle p^{-n}T\rangle\]
Similarly, for any $n'>n>0$, we have natural identification of Fr\'echet algebras
\[\eta_{n,n',A}:S_{n,n'}\ho_{\Qp}A\overset{\sim}{\longrightarrow}A\langle p^{-n}T, p^{n'}T^{-1}\rangle.\]
Moreover, the following diagram commutes
\[\xymatrix@=45pt{
S_n\ho_{\Qp}A\ar[r]^{\eta_{n,A}} \ar@{^{(}->}[d] & A\langle p^{-n}T\rangle\ar@{^{(}->}[d]\\
S_{n,n'}\ho_{\Qp}A \ar[r]^{\eta_{n,n',A}} & A\langle p^{-n}T, p^{n'}T^{-1} \rangle
}
\]
where the vertical arrows are natural inclusions.
\end{lemma}

\begin{proof}
Apply the result in the previous lemma to the $\Qp$-Banach algebras $A_i$ and then take inverse limits.
\end{proof}

In particular, Lemma~\ref{banach} applies to $A=\bcris^+$ and Lemma~\ref{frechet} applies to $A=\bdR^+=\varprojlim \bdR^+/(t^i)$. To prove Lemma~\ref{keylemma}, it remains to prove the following result.

\begin{lemma}\label{lastlemma}
\begin{enumerate}
\item[(i)] For any $M>0$, the natural continuous maps $\bcris^+\rightarrow \bdR^+/(t^i)$ induces a natural inclusion $\bcris^+\langle p^{-M}T\rangle \hookrightarrow \bdR^+\langle p^{-M}T\rangle$.
\item[(ii)] Similarly, for any $M>0$, the natural continuous maps $\bcris^+\rightarrow \bdR^+/(t^i)$ induces a natural inclusion $\bcris^+\langle p^{-M}T, p^{M+1}T^{-1}\rangle \hookrightarrow \bdR^+\langle p^{-M}T, p^{M+1}T^{-1}\rangle$.
\item[(iii)] We have the following commutative diagram

\[\xymatrix@C=90pt@R=50pt{
S_{M, M+1}\ho_{\Qp}\bcris^+\ar[r]^{\eta_{M, M+1, \bcris^+}}\ar[d]^{(\ast)} & \bcris^+\langle p^{-M}T, p^{M+1}T^{-1} \rangle\ar@{^{(}->}[d]\\
S_{M, M+1}\ho_{\Qp}\bdR^+\ar[r]^{\eta_{M, M+1, \bdR^+}} & \bdR^+\langle p^{-M}T, p^{M+1}T^{-1}\rangle \\
S_M\ho_{\Qp}\bdR^+ \ar@{^{(}->}[u] \ar[r]^{\eta_{M, \bdR^+}}& \bdR^+\langle p^{-M}T\rangle \ar@{^{(}->}[u]
}
\]

In particular, the natural map $(\ast)$ is injective.
\end{enumerate}
\end{lemma}

\begin{proof}
\begin{enumerate}
\item[(i)] For each $i$, the map $\bcris^+\rightarrow\bdR^+/(t^i)$ is continuous with respect to the Banach topology on both sides. It is easy to check that the induced map $\bcris^+\langle p^{-M}T\rangle\rightarrow(\bdR^+/(t^i))\langle p^{-M}T\rangle$ is also continuous. This gives a map $\bcris^+\langle p^{-M}T\rangle \hookrightarrow \bdR^+\langle p^{-M}T\rangle$. By identifying $\bdR^+\langle p^{-M}T\rangle$ as a subset of $\bdR^+[[T]]$ and making use of the fact that $\bcris^+\rightarrow\bdR^+$ is injective, we know that the desired map is also injective.
\item[(ii)] The proof is similar to (i).
\item[(iii)] The diagram is obviously commutative.
\end{enumerate}
\end{proof}

Now we can finish the proof of Lemma~\ref{keylemma}.

According to Lemma~\ref{lastlemma} (iii), taking intersection of $S_{M, M+1}\ho_{\Qp}\bcris^+$ and $S_M \ho_{\Qp}\bdR^+$ inside of $S_{M, M+1}\ho_{\Qp}\bdR^+$ is the same as taking the intersection of $\bcris^+\langle p^{-M}T, p^{M+1}T^{-1}\rangle$ and $\bdR^+\langle p^{-M}T\rangle$ inside of $\bdR^+\langle p^{-M}T, p^{M+1}T^{-1}\rangle\subset\bcris^+[[T, T^{-1}]]$. By (i), this intersection is $\bcris^+\langle p^{-M}T\rangle$.

\end{document}